\documentclass[12pt]{amsart}
\usepackage{amsmath,amssymb,amsthm,amscd,amsfonts,geometry}

\theoremstyle{plain}
\newtheorem{thm}{Theorem}[section]
\newtheorem{prop}[thm]{Proposition}
\newtheorem{cor}[thm]{Corollary}
\newtheorem{lem}[thm]{Lemma}

\newtheorem{main}{Main Theorem}

\theoremstyle{remark}

\theoremstyle{definition}

\newcommand{\I}{\mathbf{I}}                       
\newcommand{\sph}{\mathbf{S}}                     
\newcommand{\ball}{\mathbf{B}}                     

\newcommand{\diam}{\operatorname{diam}}           

\newcommand{\fin}{\operatorname{Fin}}             
\newcommand{\cpt}{\operatorname{Comp}}             

\title[On a hyperspace of compact subsets]{On a hyperspace of compact subsets which is homeomorphic to a non-separable Hilbert space}

\author[K.~Koshino]{Katsuhisa Koshino}
\thanks{This work was supported by JSPS KAKENHI Grant Number 15K17530}

\address{Faculty of Engineering, Kanagawa University, Yokohama, 221-8686, Japan}

\email{pt120401we@kanagawa-u.ac.jp}



\begin{document}

\begin{abstract}
Let $X$ be a metrizable space and $\cpt(X)$ be the hyperspace consisting of non-empty compact subsets of $X$ endowed with the Vietoris topology.
In this paper, we give a necessary and sufficient condition on $X$ for $\cpt(X)$ to be homeomorphic to a non-separable Hilbert space.
Moreover, we consider the topological structure of pair $(\cpt(\overline{X}),\fin(X))$ of hyperspaces of $X$ and its completion $\overline{X}$,
 where $\fin(X)$ is the hyperspace of non-empty finite sets in $X$.
\end{abstract}

\maketitle

\section{Introduction}

Throughout this paper, spaces are metrizable, maps are continuous and $\kappa$ is an infinite cardinal.
Given a space $X$, let $\cpt(X)$ be the hyperspace of non-empty compact sets in $X$ with the Vietoris topology.
The hyperspace $\cpt(X)$ is a classical object and has been studied in infinite-dimensional topology, see, for instance, \cite{Cu1,Cu2}.
D.W.~Curtis \cite{Cu1} gave a necessary and sufficient condition on $X$ for $\cpt(X)$ to be homeomorphic to the separable Hilbert space as follows:

\begin{thm}\label{l2}
A space $X$ is separable, connected, locally connected, topologically complete and nowhere locally compact if and only if $\cpt(X)$ is homeomorphic to the separable Hilbert space.
\end{thm}

We denote the Hilbert space of density $\kappa$ by $\ell_2(\kappa)$.
For a space $X$, an \textit{$X$-manifold} is a topological manifold modeled on $X$.
In the non-separable case, combining the result of \cite{Yag} with the open embedding theorem and the classification theorem of Hilbert manifolds, refer to \cite{HS},
 we can establish the following:

\begin{thm}
If a space $X$ is a connected $\ell_2(\kappa)$-manifold,
 then $\cpt(X)$ is homeomorphic to $\ell_2(\kappa)$.
\end{thm}

In this paper, we characterize a space $X$ whose hyperspace $\cpt(X)$ is homeomorphic to a Hilbert space of density $\kappa$ as follows:

\begin{main}
A space $X$ is connected, locally connected, topologically complete, nowhere locally compact, and for each point $x \in X$, any neighborhood of $x$ in $X$ is of density $\kappa$ if and only if $\cpt(X)$ is homeomorphic to $\ell_2(\kappa)$.
\end{main}

Let $\fin(X) \subset \cpt(X)$ be the hyperspace of non-empty finite subsets of a space $X$.
By $\ell_2^f(\kappa)$, we mean the linear subspace spanned by the canonical orthonormal basis of $\ell_2(\kappa)$.
D.W.~Curtis and N.T.~Nhu \cite{CN} characterized a space $X$ whose hyperspace $\fin(X)$ is homeomorphic to $\ell_2^f(\omega)$,
 and the author \cite{Kos8} generalized it as follows:

\begin{thm}
A space $X$ is connected, locally path-connected, strongly countable-dimensional, $\sigma$-locally compact and for every point $x \in X$, any neighborhood of $x$ in $X$ is of density $\kappa$ if and only if $\fin(X)$ is homeomorphic to $\ell_2^f(\kappa)$.
\end{thm}

For spaces $X$ and $Y$, writing $(X,Y)$, we understand $Y$ is a subspace of $X$.
A pair $(X,Y)$ of spaces is homeomorphic to $(X',Y')$ if there exists a homeomorphism $f : X \to X'$ such that $f(Y) = Y'$.
A subset $A$ of a space $X$ is \textit{locally non-separating} in $X$ if for every non-empty connected open set $U$ in $X$, $U \setminus A$ is non-empty and connected.
As a corollary of the main theorem and the paper \cite{Kos8}, we can establish the following:

\begin{cor}\label{pair}
Let $X$ be a connected, locally path-connected, strongly countable-dimensional and $\sigma$-locally compact space such that for each point $x \in X$, any neighborhood of $x$ in $X$ is of density $\kappa$.
Suppose that $X$ has a locally connected and nowhere locally compact completion $\overline{X}$.
Then the pair $(\cpt(\overline{X}),\fin(X))$ is homeomorphic to $(\ell_2(\kappa),\ell_2^f(\kappa))$ if and only if $\overline{X} \setminus X$ is locally non-separating in $\overline{X}$.
\end{cor}

\section{Notation and Toru\'nczyk's characterization of Hilbert manifolds}

In this section, we fix some notation and introduce Toru\'nczyk's characterization of Hilbert manifolds that will be used to prove the main theorem.
We denote the closed unit interval $[0,1]$ by $\I$.
Let $X = (X,d)$ be a metric space.
For a point $x \in X$ and subsets $A, B \subset X$, we define the distance $d(x,A)$ between $x$ and $A$ by $d(x,A) = \inf\{d(x,a) \mid a \in A\}$ and the distance $d(A,B)$ between $A$ and $B$ by $d(A,B) = \inf\{d(a,b) \mid a \in A, b \in B\}$.
For $\epsilon > 0$, let $B_d(x,\epsilon) = \{x' \in X \mid d(x,x') < \epsilon\}$ and $N_d(A,\epsilon) = \{x' \in X \mid d(x',A) < \epsilon\}$.
The diameter of $A \subset X$ is denoted by $\diam_d{A}$.
The topology of $\cpt(X)$ is induced by \textit{the Hausdorff metric} $d_H$ defined as follows:
 $$d_H(A,B) = \inf\{r > 0 \mid A \subset N_d(B,r), B \subset N_d(A,r)\}.$$

It is said that a space $X$ has \textit{the countable locally finite approximation property} if for each open cover $\mathcal{U}$ of $X$, there exists a sequence $\{f_n : X \to X\}_{n < \omega}$ of maps such that every $f_n$ is $\mathcal{U}$-close to the identity map on $X$ and the family $\{f_n(X)\}_{n < \omega}$ of the images is locally finite in $X$.
Recall that for maps $f : X \to Y$ and $g : X \to Y$, and for an open cover $\mathcal{U}$ of $Y$, $f$ is \textit{$\mathcal{U}$-close} to $g$ if for each point $x \in X$, there exists a member $U \in \mathcal{U}$ such that $f(x)$ and $g(x)$ are contained in $U$.
For $n < \omega$, a space $X$ has \textit{the $\kappa$-discrete $n$-cells property} provided that the following condition holds:
\begin{itemize}
 \item For every open cover $\mathcal{U}$ of $X$ and every map $f : \bigoplus_{\gamma < \kappa} A_\gamma \to X$,
  where each $A_\gamma = \I^n$,
  there is a map $g : \bigoplus_{\gamma < \kappa} A_\gamma \to X$ such that $g$ is $\mathcal{U}$-close to $f$ and $\{g(A_\gamma)\}_{\gamma < \kappa}$ is discrete in $X$.
\end{itemize}
H.~Toru\'nczyk \cite{Tor5,Tor6} gave the following celebrated characterization to an $\ell_2(\kappa)$-manifold (cf.~\cite[Theorem~3.1]{BZ}):

\begin{thm}\label{DCP-char.}
A connected space $X$ is an $\ell_2(\kappa)$-manifold if and only if the following conditions are satisfied:
\begin{enumerate}
 \item $X$ is a topologically complete ANR of density $\kappa$;
 \item $X$ has the countable locally finite approximation property;
 \item $X$ has the $\kappa$-discrete $n$-cells property for every $n < \omega$.
\end{enumerate}
\end{thm}

In the above, replacing ``ANR'' with ``AR'', we can obtain a characterization of a Hilbert space $\ell_2(\kappa)$.

\section{Basic topological properties of $\cpt(X)$}

In this section, some basic properties of $\cpt(X)$ are collected.
On the density of $\cpt(X)$, the following holds \cite[Corollary~4.2]{Yag}:

\begin{prop}\label{dens.}
The hyperspace $\cpt(X)$ has the same density as a space $X$.
\end{prop}

\begin{prop}\label{nbd.dens.}
Let $X$ be a space and $x \in X$.
If any neighborhood of $\{x\}$ in $\cpt(X)$ is of density $\kappa$,
 then every neighborhood of $x$ in $X$ is of density $\kappa$.
\end{prop}

\begin{proof}
Since $\cpt(X)$ is of density $\kappa$,
 $X$ is also of density $\kappa$ by Proposition~\ref{dens.}.
Suppose that the point $x \in X$ has a neighborhood $U$ of density $< \kappa$.
Observe that $\cpt(U) \subset \cpt(X)$ is a neighborhood of $\{x\}$.
Due to Proposition~\ref{dens.}, the density of $\cpt(U)$ is less than $\kappa$,
 which is a contradiction.
Hence all neighborhoods of $x$ are of density $\kappa$.
\end{proof}

We can easily observe the following (cf.~\cite[Theorem~5.12.5.~(2)]{Sa6}):

\begin{prop}\label{compl.}
For every space $X$, $X$ is topologically complete if and only if $\cpt(X)$ is topologically complete.
\end{prop}

Concerning the ANR-property of $\cpt(X)$, the following holds \cite{Wo,Ta} (cf.~\cite[Theorem~1.6]{Cu2}):

\begin{prop}\label{ar}
Let $X$ be topologically complete.
Then $X$ is connected and locally connected if and only if $\cpt(X)$ is an AR.
\end{prop}

\section{The countable locally finite approximation property of $\cpt(X)$}

In this section, using the similar strategy in the proof of Theorem~E of \cite{Cu1}, we shall verify the countable locally finite approximation property of the hyperspace $\cpt(X)$.
Let $K$ be a simplicial complex.
We denote the polyhedron\footnote{In this paper, we do not need polyhedra to be metrizable.} of $K$ by $|K|$ and the $n$-skeleton of $K$ by $K^{(n)}$ for each $n < \omega$.
We often regard $\sigma \in K$ as a simplicial complex consisting of its faces.
The next two lemmas concerning nice subdivisions of simplicial complexes are used in the proof of Theorem~E of \cite{Cu1} and the details of the proofs is given in \cite{Kos8}.

\begin{lem}\label{subd.1}
Let $X = (X,d)$ be a metric space, $K$ be a simplicial complex and $f : |K| \to X$ be a map.
For every map $\alpha : X \to (0,\infty)$, there is a subdivision $K'$ of $K$ such that $\diam_d{f(\sigma)} < \inf_{x \in \sigma} \alpha f(x)$ for any $\sigma \in K'$.
\end{lem}

\begin{lem}\label{subd.2}
For each map $\alpha : |K| \to (0,\infty)$ of the polyhedron of a simplicial complex $K$ and $\beta > 1$, $K$ has a subdivision $K'$ such that $\sup_{x \in \sigma} \alpha(x) < \beta\inf_{x \in \sigma} \alpha(x)$ for all $\sigma \in K'$.
\end{lem}

The following two lemmas are also used in the proof of Theorem~E of \cite{Cu1}.
The analogues of these lemmas for hyperspaces of finite subsets are proved in \cite{Kos8}.

\begin{lem}\label{subseq.}
Let $X = (X,d)$ be a metric space.
Suppose that $\{A_n\}_{n < \omega}$ is a sequence in $\cpt(X) = (\cpt(X),d_H)$ converging to $A \in \cpt(X)$.
Then for each closed set $B_n \subset A_n$, $\{B_n\}_{n < \omega}$ has a subsequence that converges to some compact subset $B \subset A$.
\end{lem}

\begin{lem}\label{arc}
Let $X = (X,d)$ be a locally path-connected metric space and $\alpha : \cpt(X) \to (0,\infty)$ be a map.
Then there is a map $\beta : \cpt(X) \to (0,\infty)$ such that for any $A \in \cpt(X)$, each point $x \in N_d(A,\beta(A))$ has an arc connecting to some point of $A$ of diameter $< \alpha(A)$.
\end{lem}

Now, we show the following:

\begin{prop}\label{lfap}
Let $X$ be a locally path-connected and nowhere locally compact space.
Then $\cpt(X)$ has the countable locally finite approximation property.
\end{prop}

\begin{proof}
Let $\mathcal{U}$ be an open cover of $\cpt(X)$.
Take an open cover $\mathcal{V}$ of $\cpt(X)$ that is a star-refinement of $\mathcal{U}$.
Since $\cpt(X)$ is an ANR by Theorem~1.6 of \cite{Cu2},
 there are a simplicial complex $K$ and maps $f : \cpt(X) \to |K|$, $g : |K| \to \cpt(X)$ such that $gf$ is $\mathcal{V}$-close to the identity map on $\cpt(X)$, refer to \cite[Theorem~6.6.2]{Sa6}.
It suffices to construct a map $g_i : |K| \to \cpt(X)$ for each $i < \omega$ so that $g_i$ is $\mathcal{V}$-close to $g$ and $\{g_i(|K|)\}_{i < \omega}$ is locally finite in $\cpt(X)$.
Then $\{g_if\}_{i < \omega}$ will be the desired sequence of maps.

Take an admissible metric $d$ on $X$ and a map $\alpha : \cpt(X) \to (0,1)$ so that the family $\{B_{d_H}(A,2\alpha(A)) \mid A \in \cpt(X)\}$ refines $\mathcal{V}$.
Since $X$ is locally path-connected,
 according to Lemma~\ref{arc}, there exists a map $\beta : \cpt(X) \to (0,1)$ such that for every $A \in \cpt(X)$, each point $x \in N_d(A,\beta(A))$ has an arc from some point of $A$ to $x$ of diameter $< \alpha(A)/2$.
We may assume that $\beta(A) \leq \alpha(A)/2$ for every $A \in \cpt(X)$.
Combining Lemmas~\ref{subd.1} with \ref{subd.2}, we can replace $K$ with a subdivision so that for every $\sigma \in K$,
\begin{enumerate}
 \item $\diam_{d_H}g(\sigma) < \inf_{y \in \sigma} \beta g(y)/2$,
 \item $\sup_{y \in \sigma} \beta g(y) < 2\inf_{y \in \sigma} \beta g(y)$,
 \item $\sup_{y \in \sigma} \alpha g(y) < 4\inf_{y \in \sigma} \alpha g(y)/3$.
\end{enumerate}
For each $n < \omega$, we can find a locally finite open cover $\mathcal{V}_n$ of $X$ of mesh $< 1/n$.
Since $X$ is nowhere locally compact,
 for every $n < \omega$ and $\emptyset \neq V \in \mathcal{V}_n$, $V$ contains an infinite subset $Z(V) = \{z_V^i \mid i < \omega\}$ that is discrete in $X$.
Let $Z^i(n) = \{z_V^i \mid V \in \mathcal{V}_n\}$, $i < \omega$.
Here we may assume that $Z^i(n) \cap Z^j(n) = \emptyset$ if $i \neq j$.
Indeed, it will be shown by induction.
Suppose that for some $i < \omega$, $Z^j(n) \cap Z^k(n) = \emptyset$ if $j < k < i$.
By the local finiteness of $\mathcal{V}_n$, for every $z_V^i \in Z^i(n)$, the family $\{V' \in \mathcal{V}_n \mid z_V^i \in V'\}$ is finite.
Since $X$ is nowhere locally compact and each $Z(V)$ is discrete,
 we can find a point $x_V^i \in V$ sufficiently close to $z_V^i$ so that $x_V^i \notin \bigcup_{j < i} Z^j(n)$ and even if $z_V^i$ is substituted by $x_V^i$,
 $Z(V)$ is still discrete.
Due to this substitution, we have that $Z^j(n) \cap Z^i(n) = \emptyset$ for all $j < i$.
Put $Z(n) = \bigcup_{\emptyset \neq V \in \mathcal{V}_n} Z(V) = \bigoplus_{i < \omega} Z^i(n)$,
 so it is locally finite in $X$.

First, we shall construct the restriction $g_i|_{K^{(0)}}$, $i < \omega$.
For every $v \in K^{(0)}$, there is $n_v \geq 2$ such that $1/n_v < \beta g(v)/4 \leq 1/(n_v - 1)$.
Then we can find a point $z^i(v) \in Z^i(n_v) \subset Z(n_v)$, $i < \omega$, so that $d(z^i(v),g(v)) < 1/n_v$.
Note that for any $v', v'' \in K^{(0)}$ with $n_{v'} = n_{v''}$, $z^i(v') \neq z^j(v'')$ if $i \neq j$.
Let $g_i(v) = g(v) \cup \{z^i(v)\} \in \cpt(X)$,
 so
 $$d_H(g(v),g_i(v)) < 1/n_v < \beta g(v)/4 \leq \alpha g(v)/2.$$

Next, we will extend each $g_i$ over the $1$-skeleton $|K^{(1)}|$.
Let $\sigma \in K^{(1)} \setminus K^{(0)}$, $\sigma^{(0)} = \{v_1, v_2\}$ and $\hat\sigma$ be its barycenter.
According to conditions (1) and (2), we get for any $m = 1, 2$,
\begin{align*}
 d(z^i(v_m),g(\hat\sigma)) &\leq d(z^i(v_m),g(v_m)) + d_H(g(v_m),g(\hat\sigma)) < \beta g(v_m)/4 + \diam_{d_H}g(\sigma)\\
 &< \sup_{y \in \sigma} \beta g(y)/4 + \inf_{y \in \sigma} \beta g(y)/2 < \inf_{y \in \sigma} \beta g(y) \leq \beta g(\hat\sigma).
\end{align*}
Applying Lemma~\ref{arc}, we can take an arc $\gamma_m : \I \to X$ from some point of $g(\hat\sigma)$ to $z^i(v_m)$ of $\diam_d{\gamma_m(\I)} < \alpha g(\hat\sigma)/2$.
Put $g_i(\hat\sigma) = g(\hat\sigma) \cup \{z^i(v_m) \mid m = 1, 2\}$.
Then $d_H(g(\hat\sigma),g_i(\hat\sigma)) \leq \beta g(\hat\sigma) \leq \alpha g(\hat\sigma)/2$.
Let $\phi : \I \to \cpt(X)$ be a map defined by $\phi(t) = g(\hat\sigma) \cup \{\gamma_m(t) \mid m = 1, 2\}$,
 which is a path from $g(\hat\sigma)$ to $g_i(\hat\sigma)$.
For each $m = 1, 2$, define $g_i : \langle v_m,\hat\sigma \rangle \to \cpt(X)$,
 where $\langle v_m,\hat\sigma \rangle$ is the segment between $v_m$ and $\hat\sigma$,
 as follows:
 $$g_i((1 - t)v_m + t\hat\sigma) = \left\{
 \begin{array}{ll}
  g((1 - 2t)v_m + 2t\hat\sigma) \cup \{z^i(v_m)\} &\text{if } 0 \leq t \leq 1/2,\\
  \phi(2t - 1) \cup \{z^i(v_m)\} &\text{if } 1/2 \leq t \leq 1.
 \end{array}
 \right.$$
Then for each $y \in \sigma$, if $y = (1 - t)v_m + t\hat\sigma$, $0 \leq t \leq 1/2$,
\begin{align*}
 d_H(g(\hat\sigma),g_i(y)) &\leq \max\{d_H(g(\hat\sigma),g((1 - 2t)v_m + 2t\hat\sigma)),d(g(\hat\sigma),z^i(v_m))\}\\
 &\leq \max\{\diam_{d_H}g(\sigma),\beta g(\hat\sigma)\} \leq \max\{\inf_{y' \in \sigma} \beta g(y')/2,\beta g(\hat\sigma)\}\\
 &\leq \beta g(\hat\sigma) \leq \alpha g(\hat\sigma)/2,
\end{align*}
 and if $y = (1 - t)v_m + t\hat\sigma$, $1/2 \leq t \leq 1$,
\begin{align*}
 d_H(g(\hat\sigma),g_i(y)) &\leq \max\{d_H(g(\hat\sigma),\phi(2t - 1)),d(g(\hat\sigma),z^i(v_m))\}\\
 &\leq \max\{\max\{\diam_d\gamma_n(\I) \mid n = 1, 2\},\beta g(\hat\sigma)\}\\
 &\leq \max\{\alpha g(\hat\sigma)/2,\beta g(\hat\sigma)\} \leq \alpha g(\hat\sigma)/2.
\end{align*}
It follows from condition (3) that
\begin{align*}
 d_H(g(y),g_i(y)) &\leq d_H(g(y),g(\hat\sigma)) + d_H(g(\hat\sigma),g_i(y)) \leq \diam_{d_H}g(\sigma) + \alpha g(\hat\sigma)/2\\
 &< \inf_{y' \in \sigma} \beta g(y')/2 + \alpha g(\hat\sigma)/2 \leq \beta g(\hat\sigma)/2 + \alpha g(\hat\sigma)/2 \leq 3\alpha g(\hat\sigma)/4\\
 &\leq 3\sup_{y' \in \sigma} \alpha g(y')/4 < \inf_{y' \in \sigma} \alpha g(y') \leq \alpha g(y).
\end{align*}
Note that each $g_i(y)$ contains $z^i(v_1)$ or $z^i(v_2)$.

By induction, we shall construct a map $g_i : |K| \to \cpt(X)$, $i < \omega$, such that for each $\sigma \in K \setminus K^{(0)}$ and each $y \in \sigma$, $g_i(y) = \bigcup_{a \in A(y)} g_i(a)$ for some $A(y) \in \fin(|\sigma^{(1)}|)$.
Assume that $g_i$ has extended over $|K^{(n)}|$ for some $n < \omega$ such that for every $\sigma \in K^{(n)} \setminus K^{(0)}$ and $y \in \sigma$, $g_i(y) = \bigcup_{a \in A(y)} g_i(a)$ for some $A(y) \in \fin(|\sigma^{(1)}|)$.
Take any $n + 1$-simplex $\sigma \in K^{(n + 1)} \setminus K^{(n)}$.
Due to Lemma~3.3 of \cite{CN}, there exists a map $r : \sigma \to \fin(\partial\sigma)$ such that $r(y) = \{y\}$ for all $y \in \partial\sigma$,
 where $\partial{\sigma}$ means the boundary of $\sigma$.
The restriction $g_i|_{\partial\sigma}$ induces $\tilde{g_i} : \fin(\partial\sigma) \to \cpt(X)$ defined by $\tilde{g_i}(A) = \bigcup_{a \in A} g_i(a)$.
Then the composition $g_{i,\sigma} = \tilde{g_i}r : \sigma \to \cpt(X)$ satisfies that $g_{i,\sigma}|_{\partial\sigma} = g_i|_{\partial\sigma}$.
Observe that for each $y \in \sigma$,
 $$g_{i,\sigma}(y) = \tilde{g_i}r(y) = \bigcup_{y' \in r(y)} g_i(y') = \bigcup_{y' \in r(y)} \bigcup_{a \in A(y')} g_i(a) = \bigcup_{a \in \bigcup_{y' \in r(y)} A(y')} g_i(a),$$
 where $g_i(y') = \bigcup_{a \in A(y')} g_i(a)$ for some $A(y') \in \fin(|\sigma^{(1)}|)$ by the inductive assumption.
Therefore $g_i$ can be extended over $|K^{(n + 1)}|$ by $g_i|_\sigma = g_{i,\sigma}$ for all $\sigma \in K^{(n + 1)} \setminus K^{(n)}$.

Completing this induction, we can obtain a map $g_i : |K| \to \cpt(X)$ for every $i < \omega$.
For each $y \in \sigma \in K \setminus K^{(0)}$ and each $a \in |\sigma^{(1)}|$, by conditions (1) and (3), we have
\begin{align*}
 d_H(g(y),g_i(a)) &\leq d_H(g(y),g(a)) + d_H(g(a),g_i(a)) < \diam_{d_H}g(\sigma) + \alpha g(a)\\
 &< \inf_{y' \in \sigma} \beta g(y')/2 + \sup_{y' \in \sigma} \alpha g(y') \leq \inf_{y' \in \sigma} \alpha g(y')/4 + 4\inf_{y' \in \sigma} \alpha g(y')/3\\
 &= 19\inf_{y' \in \sigma} \alpha g(y')/12 < 2\alpha g(y).
\end{align*}
It follows that
 $$d_H(g(y),g_i(y)) = d_H\Bigg(g(y),\bigcup_{a \in \bigcup_{y' \in r(y)} A(y')} g_i(a)\Bigg) \leq \max_{a \in \bigcup_{y' \in r(y)} A(y')} d_H(g(y),g_i(a)) < 2\alpha g(y),$$
 which implies that $g_i$ is $\mathcal{V}$-close to $g$.
Remark that $z^i(v) \in g_i(y)$ for some vertex $v \in \sigma^{(0)}$.
Here we may replace $g_i(y)$ with the union $g(y) \cup g_i(y)$ for every $y \in |K|$,
 so we get $g(y) \subset g_i(y)$.
It remains to prove that $\{g_i(|K|)\}_{i < \omega}$ is locally finite in $\cpt(X)$.

Suppose the contrary.
Then we can find a subsequence $\{g_{n_i}\}_{i < \omega}$ of $\{g_i\}_{i < \omega}$ such that $\{g_{n_i}(y_i)\}_{i < \omega}$, $y_i \in |K|$, is converging to some $A \in \cpt(X)$.
For simplicity, replace each $g_{n_i}$ with $g_i$.
Take the carrier $\sigma_i \in K$ of $y_i$ and choose a vertex $v_i \in \sigma_i^{(0)}$ so that $z^i(v_i) \in g_i(y_i)$.
Since $g(y_i) \subset g_i(y_i)$,
 replacing $\{g(y_i)\}_{i < \omega}$ with a subsequence, we can obtain a compact subset $B \subset A$ to which $\{g(y_i)\}_{i < \omega}$ converges by Lemma~\ref{subseq.}.
Thus $\{\beta g(y_i)\}_{i < \omega}$ converges to $\beta(B) > 0$.
On the other hand, since $\{g_i(y_i)\}_{i < \omega}$ converges to $A$,
 any subsequence of $\{z^i(v_i)\}_{i < \omega}$ has an accumulation point in $A$.
Then $\{n_{v_i}\}_{i < \omega}$ diverges to $\infty$.
Indeed, supposing the contrary, we can find $n_0 < \omega$ and replace $\{n_{v_i}\}_{i < \omega}$ with a subsequence so that $n_{v_i} = n_0$ for all $i < \omega$.
By the choice of $z^i(v_i)$, $\{z^i(v_i)\}_{i < \omega}$ is pairwise distinct and contained in the locally finite subset $Z(n_0)$,
 which is a contradiction.
Hence $\{\beta g(v_i)\}_{i < \omega}$ converges to $0$ because $1/n_{v_i} < \beta g(v_i)/4 \leq 1/(n_{v_i} - 1)$.
Moreover, by condition (2), for every $i < \omega$,
 $$\beta g(y_i) \leq \sup_{y \in \sigma_i} \beta g(y) < 2\inf_{y \in \sigma_i} \beta g(y) \leq 2\beta g(v_i),$$
 so $\{\beta g(y_i)\}_{i < \omega}$ also converges to $0$.
This is a contradiction.
Consequently, $\{g_i(|K|)\}_{i < \omega}$ is locally finite in $\cpt(X)$.
\end{proof}

\section{The $\kappa$-discrete $n$-cells property of $\cpt(X)$}

This section is devoted to the verification of the $\kappa$-discrete $n$-cells property in $\cpt(X)$.
To detect the $\kappa$-discrete $n$-cells property in a space, the following two lemmas are useful.

\begin{lem}\cite[Lemma~3.1]{BZ}\label{loc.fin.}
Let $n < \omega$.
A space $X$ has the $\kappa$-discrete $n$-cells property if and only if for each open cover $\mathcal{U}$ of $X$, and each map $f : \bigoplus_{\gamma < \kappa} A_\gamma \to X$, where each $A_\gamma = \I^n$,
 there is a map $g : \bigoplus_{\gamma < \kappa} A_\gamma \to X$ such that $g$ is $\mathcal{U}$-close to $f$ and $\{g(A_\gamma)\}_{\gamma < \kappa}$ is locally finite in $X$.
\end{lem}

\begin{lem}\cite[Lemma~3.2]{BZ}\label{cof.}
Let $X$ be a space with the countable locally finite approximation property and $n < \omega$.
The space $X$ has the $\kappa$-discrete $n$-cells property if and only if it has the $\lambda$-discrete $n$-cells property for every $\lambda \leq \kappa$ of uncountable cofinality.
\end{lem}

The following lemma can be easily observed, refer to the proof of \cite[Lemma~6.2]{BZ}:

\begin{lem}\label{discr.}
Let $X$ be a space and $\kappa$ be of uncountable cofinality.
If a subset $A \subset X$ is of density $\geq \kappa$,
 then $A$ contains a discrete subset of cardinality $\geq \kappa$.
\end{lem}

Now, we get the following:

\begin{prop}\label{dcp}
Let $X$ be locally path-connected and nowhere locally compact.
Suppose that any neighborhood of each point in $X$ is of density $\geq \kappa$.
Then the hyperspace $\cpt(X)$ has the $\kappa$-discrete $n$-cells property for every $n < \omega$.
\end{prop}

\begin{proof}
According to Proposition~\ref{lfap}, the hyperspace $\cpt(X)$ has the countable locally finite approximation property.
Hence Lemma~\ref{cof.} guarantees that we may only consider $\kappa$ be of uncountable cofinality.
Let $n < \omega$ and $\mathcal{V}$ be an open cover of $\cpt(X)$.
By virtue of Lemma~\ref{loc.fin.}, it is sufficient to show that for any map $g : \bigoplus_{\gamma < \kappa} A_\gamma \to \cpt(X)$, where each $A_\gamma = \I^n$,
 there is a map $g_\gamma : A_\gamma \to \cpt(X)$, $\gamma < \kappa$, such that $g_\gamma$ is $\mathcal{V}$-close to $g|_{A_\gamma}$ and $\{g_\gamma(A_\gamma)\}_{\gamma < \kappa}$ is locally finite in $\cpt(X)$.

Take an admissible metric $d$ on $X$ and the same maps $\alpha, \beta : \cpt(X) \to (0,1)$ as in Proposition~\ref{lfap}.
Moreover, combining Lemmas~\ref{subd.1} with \ref{subd.2}, we can triangulate each $A_\gamma$ into a simplicial complex $K_\gamma$ satisfying the same conditions (1), (2) and (3) as in Proposition~\ref{lfap}.
For each $m < \omega$, choose a locally finite open cover $\mathcal{V}_m$ of $X$ of mesh $< 1/m$.
By the assumption and Lemma~\ref{discr.}, for each $m < \omega$ and each non-empty $V \in \mathcal{V}_m$, there is a discrete subset $Z(V) \subset V$ of cardinality $\geq \kappa$.
By the local finiteness of $\mathcal{V}_m$, the subset $Z(m) = \bigcup_{\emptyset \neq V \in \mathcal{V}_m} Z(V)$ is locally finite in $X$.

We construct the desired map $g_\gamma$ for each $\gamma < \kappa$.
For every $v \in K_\gamma^{(0)}$, there is $m_v \geq 2$ such that $1/m_v < \beta g(v)/4 \leq 1/(m_v - 1)$.
Then we can choose a point $z^\gamma(v) \in Z(m_v)$, $\gamma < \kappa$, so that $d(z^\gamma(v),g_\gamma(v)) < 1/m_v$ and for any $v', v'' \in \bigcup_{\gamma < \kappa} K_\gamma^{(0)}$ with $m_{v'} = m_{v''}$, $z^\gamma(v') \neq z^{\gamma'}(v'')$ if $\gamma < \gamma' < \kappa$,
 where we may adjust each $Z(m_v)$ as in the proof of Proposition~\ref{lfap} if necessary.
Define $g_\gamma(v) = g(v) \cup \{z^\gamma(v)\} \in \cpt(X)$.
The rest of the proof follows from the same argument as in Proposition~\ref{lfap}.
\end{proof}

\section{Proof of the main theorem}

In this final section, we shall prove the main theorem.

\begin{proof}[Proof of the main theorem]
In the case that $X$ is separable,
 the proof follows from Theorem~\ref{l2}.
Let $\kappa$ be uncountable.

(The ``only if'' part)~According to Propositions~\ref{dens.}, \ref{compl.} and \ref{ar}, the hyperspace $\cpt(X)$ is a topologically complete AR of density $\kappa$.
Since $X$ is connected, locally connected and topologically complete,
 it is locally path-connected.
Due to Proposition~\ref{lfap}, $\cpt(X)$ has the countable locally finite approximation property.
Moreover, the $\kappa$-discrete $n$-cells property, $n < \omega$, of $\cpt(X)$ follows from Proposition~\ref{dcp}.
Using Toru\'nczyk's characterization~\ref{DCP-char.}, we have that $\cpt(X)$ is homeomorphic to $\ell_2(\kappa)$.

(The ``if'' part)~Since $\cpt(X)$ is homeomorphic to $\ell_2(\kappa)$,
 it is a topologically complete AR,
 and hence, by Propositions~\ref{compl.} and \ref{ar}, $X$ is connected, locally connected and topologically complete.
Remark that for each $A \in \cpt(X)$, all neighborhoods of $A$ are of density $\kappa$.
It follows from Proposition~\ref{nbd.dens.} that any neighborhood of each point in $X$ is also of density $\kappa$,
 and hence $X$ is nowhere locally compact.
Thus the proof is complete.
\end{proof}

\section{Pair of hyperspaces}

In this final section, we will discuss the topological structure of pair of hyperspaces.
A subset $A$ of a space $X$ is said to be \textit{homotopy dense} in $X$ if there exists a homotopy $h : X \times \I \to X$ such that $h(x,0) = x$ for all $x \in X$ and $h(X \times (0,1]) \subset A$.
To show Corollary~\ref{pair}, we will use the following characterization of the pair $(\ell_2(\kappa),\ell_2^f(\kappa))$ \cite{We1,Kos1}:

\begin{thm}\label{char.pair}
A pair $(X,Y)$ of spaces is homeomorphic to $(\ell_2(\kappa),\ell_2^f(\kappa))$ if and only if $X$ is homeomorphic to $\ell_2(\kappa)$, $Y$ is homeomorphic to $\ell_2^f(\kappa)$ and $Y$ is homotopy dense in $X$.
\end{thm}

We denote the $n$-dimensional unit sphere by $\sph^n$ and the $n$-dimensional unit ball by $\ball^n$.
The homotopy density between ANRs is characterized as follows \cite[Corollary~7.4.6]{Sa6}:

\begin{lem}\label{homot.dens.}
Suppose that $X$ and $Y$ are ANRs and $Y$ is dense in $X$.
Then $Y$ is homotopy dense in $X$ if and only if the following condition holds:
\begin{itemize}
 \item For each point $x \in X$ and each neighborhood $U$ of $x$ in $X$, there is a neighborhood $V \subset U$ of $x$ such that any map $f : \sph^n \to V \cap Y$ can extend to a map $\tilde{f} : \ball^{n + 1} \to U \cap Y$ for all $n < \omega$.
\end{itemize}
\end{lem}

Using this lemma, we shall prove Corollary~\ref{pair}:

\begin{proof}[Proof of Corollary~\ref{pair}]
The main theorems of this paper and the paper \cite{Kos8} guarantee that $\cpt(\overline{X})$ is homeomorphic to $\ell_2(\kappa)$ and $\fin(X)$ is homeomorphic to $\ell_2^f(\kappa)$.
By virtue of Theorem~\ref{char.pair}, it remains to prove that $\fin(X)$ is homotopy dense in $\cpt(\overline{X})$ if and only if $\overline{X} \setminus X$ is locally non-separating in $\overline{X}$.

The ``only if'' part follows from the same argument as the proof of implication $({\rm ii}) \Rightarrow ({\rm iii})$ in \cite[Theorem~3.2]{Cu4}.
We shall prove the ``if'' part.
As is easily observed,
 $\fin(X)$ is dense in $\cpt(\overline{X})$.
Due to Lemma~\ref{homot.dens.}, we need only to show that for each point $A \in \cpt(\overline{X})$ and each neighborhood $\mathcal{U}$ of $A$ in $\cpt(\overline{X})$, there is a neighborhood $\mathcal{V} \subset \mathcal{U}$ of $A$ such that any map $f : \sph^n \to \mathcal{V} \cap \fin(X)$ can extend to a map $\tilde{f} : \ball^{n + 1} \to \mathcal{U} \cap \fin(X)$ for all $n < \omega$.
Let $A \in \cpt(\overline{X})$ and $\mathcal{U}$ be a neighborhood of $A$ in $\cpt(\overline{X})$.
By the local connectedness of $\overline{X}$ and the compactness of $A$, there exist a finite number of points $a_1, \cdots, a_n \in A$ and connected open neighborhoods $U_i$ of $a_i$, $1 \leq i \leq n$, such that $A \subset \bigcup_{i = 1}^n U_i$ and for any $B \in \cpt(\overline{X})$, $B \in \mathcal{U}$ if $B \subset \bigcup_{i = 1}^n U_i$ and $B \cap U_i \neq \emptyset$ for all $i \in \{1, \cdots, n\}$.
Let
 $$\mathcal{V} = \bigg\{B \in \cpt(\overline{X}) \biggm| B \subset \bigcup_{i = 1}^n U_i \text{ and } B \cap U_i \neq \emptyset \text{ for every } 1 \leq i \leq n\bigg\}.$$
To show that $\mathcal{V}$ is the desired neighborhood of $A$, fix an arbitrary map $f : \sph^n \to \mathcal{V} \cap \fin(X)$.

(1)~$n = 0$.
Note that $\sph^0 = \{0,1\}$.
It suffices to construct a path between $f(0)$ and $f(1)$ in $\mathcal{V} \cap \fin(X)$.
Since $f(0), f(1) \in \mathcal{V} \cap \fin(X)$,
 for each $x \in f(1)$, we can choose $i(x) \in \{1, \cdots, n\}$ and $y(x) \in f(0)$ so that $x, y(x) \in U_{i(x)} \cap X$.
Since $\overline{X} \setminus X$ is locally non-separating in $\overline{X}$,
 $U_{i(x)} \cap X$ is connected.
Hence the open subset $U_{i(x)} \cap X$ is path-connected because $X$ is locally path-connected.
So there is a path $\gamma(x)$ from $y(x)$ to $x$ in $U_{i(x)} \cap X$.
Then we can define a path $\phi : \I \to \fin(X)$ from $f(0)$ to $f(0) \cup f(1)$ by $\phi(t) = f(0) \cup \{\gamma(x)(t) \mid x \in f(1)\}$.
Observe that $\phi(\I) \subset \mathcal{V}$.
By the same argument, we can obtain a path $\psi : \I \to \mathcal{V} \cap \fin(X)$ from $f(1)$ to $f(0) \cup f(1)$.
Join the paths $\phi$ and $\psi$,
 so we can take a path between $f(0)$ and $f(1)$ that is contained in $\mathcal{V} \cap \fin(X)$.

(2)~$n \geq 1$.
The map $f$ induces $\bar{f} : \fin(\sph^n) \to \fin(X)$ defined by $\bar{f}(B) = \bigcup_{b \in B} f(b)$.
Due to Lemma~3.3 of \cite{CN}, there exists a map $r : \ball^{n + 1} \to \fin(\sph^n)$ such that $r(y) = \{y\}$ for all $y \in \sph^n$.
Let $\tilde{f} = \bar{f}r$,
 that is the desired extension of $f$.
Indeed, for each $x \in \sph^n$, we have
 $$\tilde{f}(x) = \bar{f}r(x) = \bar{f}(\{x\}) = f(x).$$
Moreover, for every $x \in \ball^{n + 1}$, it follows that
 $$\tilde{f}(x) = \bar{f}r(x) = \bigcup_{b \in r(x)} f(b) \in \mathcal{V} \cap \fin(X)$$
 because each $f(b) \in \mathcal{V} \cap \fin(X)$.
Consequently, $\mathcal{V}$ is the desired neighborhood,
 so $\fin(X)$ is homotopy dense in $\cpt(\overline{X})$.
The proof is complete.
\end{proof}

\end{document}